\documentclass[11 pt]{amsart}
\usepackage{fancyhdr}
\usepackage{amsmath}
\usepackage{amssymb}
\usepackage{amsthm}
\usepackage{bbm}
\usepackage{verbatim}
\usepackage{latexsym}
\usepackage{color}
\usepackage{mathdots}
\usepackage{xypic}
\usepackage[margin=1.25in]{geometry}
\usepackage{enumitem}
\usepackage{stmaryrd}
\usepackage{esint}
\usepackage{cite}
\usepackage{hyperref}
\usepackage{mathtools}

\newcommand{\bb}[1]{\mathbb{#1}}
\newcommand{\cc}[1]{\mathcal{#1}}

\numberwithin{equation}{section}
\newtheorem{theorem}{Theorem}
\newtheorem{corollary}[theorem]{Corollary}
\newtheorem{lemma}[theorem]{Lemma}
\newtheorem{prop}[theorem]{Proposition}
\numberwithin{theorem}{section}

\renewcommand\tilde{\widetilde}

\begin{document}

\title{A Bound on the Cohomology of Quasiregularly Elliptic Manifolds}
\author{Eden Prywes}
\address{Department of Mathematics, University of California, Los Angeles}
\email{eprywes@math.ucla.edu}
\date{\today}
\maketitle
\begin{abstract}
We show that a closed, connected and orientable Riemannian manifold of dimension $d$ that admits a quasiregular mapping from $\bb R^d$ must have bounded cohomological dimension independent of the distortion of the map.
The dimension of the degree $l$ de Rham cohomology of $M$ is bounded above by $\binom{d}{l}$.  This is a sharp upper bound that proves the Bonk-Heinonen conjecture \cite{bonkheinonen}.
A corollary of this theorem answers an open problem posed by Gromov in 1981 \cite{gromov1981}.  He asked whether there exists a $d$-dimensional, simply connected manifold that does not admit a quasiregular map from $\bb R^d$.  Our result gives an affirmative answer to this question.
\end{abstract}

\section{Introduction}
Let $M$ be a closed, connected and orientable Riemannian manifold of dimension $d$.
A $K$-\textit{quasiregular mapping}, $K\ge 1$, is a continuous mapping $f\colon \bb R^d \to M$ such that $f\in W_{\text{loc}}^{1,d}(\bb R^d,M)$ and the differential, $Df : T\bb R^d \to TM$, satisfies
\begin{align*}
    \|Df(x)\|^d \le KJ_f(x)
\end{align*}
for almost every $x\in \bb R^d$, where $J_f = \det(Df)$.  If $M$ admits such a nonconstant quasiregular mapping then we call $M$ \textit{quasiregularly elliptic}.
The main result of this paper is as follows:
\begin{theorem}\label{mainthm}
Let $M$ be a closed, connected and orientable Riemannian manifold of dimension $d$.  If $M$ admits a nonconstant quasiregular mapping from $\bb R^d$, then $\dim H^l(M) \le \binom{d}{l}$, for $0 \le l \le d$, where $H^l(M)$ is the de Rham cohomology of $M$ of degree $l$.
\end{theorem}

Theorem \ref{mainthm} is the first result that gives a restriction, independent of the fundamental group of $M$ and the distortion $K$ of the mapping, on quasiregular ellipticity of manifolds.  A $K$-dependent version of Theorem \ref{mainthm} was proved by Bonk and Heinonen \cite{bonkheinonen}.  They showed that $\dim H^l(M) \le C(d,l,K)$ and conjectured that the constant is independent of $K$.  Theorem \ref{mainthm} answers this with a sharp bound. The $d$-dimensional torus, $T^d = S^1\times \cdots \times S^1$, is quasiregularly elliptic and $\dim H^l(T^d) = \binom{d}{l}$.

This theorem also gives an answer to a longstanding open problem first posed by Gromov in 1981 \cite[p.\ 200]{gromov1981}. He asked whether their exists a $d$-dimensional, simply connected manifold that does not admit a nonconstant quasiregular mapping from $\bb R^d$.
Theorem \ref{mainthm} implies the following corollary.
\begin{corollary}
The simply connected manifold $M = \#^n (S^2\times S^2)$, the connected sum of $n$ copies of $S^2\times S^2$, is not quasiregularly elliptic for $n \ge 4$.
\end{corollary}
\begin{proof}
Firstly, the 2-sphere $S^2$, and hence $S^2\times S^2$, is simply connected.  Furthermore, since the dimension is larger than $2$, the connected sum of simply connected manifolds is simply connected.  So $M$ is simply connected.

The sphere $S^2$ satisfies $\dim H^2(S^2) = 1$.  By the K\"unneth formula \cite[p. 47]{botttu}, $\dim H^2(S^2\times S^2) = 2$.  For $1 \le l \le d-2$, $H^l(M\#N) \cong H^l(M) \oplus H^l(N)$, whenever $M$ and $N$ are smooth manifolds by the Mayer-Vietoris Theorem \cite[p. 22]{botttu}.  Therefore $\dim H^2(M) = 2n > \binom{4}{2}$.  So by Theorem \ref{mainthm}, $M$ is not quasiregularly elliptic.
\end{proof}

Theorem \ref{mainthm} is a generalization of a classical theorem for holomorphic functions in dimension 2.  Let $M$ be a Riemann surface, by the uniformization theorem, the universal cover of $M$ is either $\widehat{\bb C},\bb C$, or $\bb D$.  If $f\colon \bb C \to M$ is holomorphic, then $f$ lifts to a holomorphic map from $\bb C$ to the universal cover of $M$.  If the universal covering space is $\bb D$, then Liouville's theorem states that $f$ is constant.  
This implies that the only compact Riemann surfaces that admit holomorphic mappings are homeomorphic to $\widehat{\bb C}$ and $S^1\times S^1$.
This proof can be applied to quasiregular mappings in dimension 2 because every quasiregular mapping $f = g \circ \phi$, where $g$ is holomorphic and $\phi \colon \bb C \to \bb C$ is a quasiconformal homeomorphism \cite[p. 247]{lehtovirtanen}.  

A $1$-quasiregular map on $\bb C$ is a holomorphic function. If we study quasiregular ellipticity for $K = 1$ in higher dimensions, then the results are as restrictive as in the $d=2$ case.    If $M$ admits a $1$-quasiregular mapping from $\bb R^d$, Bonk and Heinonen \cite[Proposition 1.4]{bonkheinonen} showed that $M$ must be a quotient of the $d$-dimensional sphere or torus.  
For manifolds of dimension 3, Theorem \ref{mainthm} is known for each $K \ge 1$.  Jormakka \cite{jormakka} showed that if $M$ is quasiregularly elliptic then $M$ must be a quotient of $S^3,T^3$, or $S^2\times S^1$.  One sees that in higher dimensions there are separate results for when $K = 1$ and when $K \ge 1$.  
In the study of $K$-quasiregular mappings for $d \ge 4$, there are very few conditions on the topology of $M$ that restrict which manifolds can be quasiregularly elliptic, independent of $K$. 

A theorem by Varopoulos gives a $K$-independent result.  It states that the polynomial order of growth of the Cayley graph of the fundamental group of a quasiregularly elliptic manifold is bounded by $d$ (see \cite[Theorem X.5.1]{vsc} or \cite[Chapter 6]{gromov2007}).  This result gives a $K$-independent bound on the size of the fundamental group of the manifold, but does not apply when the fundamental group is small, specifically when the manifold is simply connected.

A recent theorem due to Kangasniemi \cite{kangasniemi} gives a $K$-independent bound on the cohomology for manifolds that admit uniformly quasiregular self-mappings.  He proved an analogue to theorem \ref{mainthm} with the added assumption that $M$ admits a non-injective quasiregular mapping $f \colon M \to M$  such that the iterates of $f$ are also $K$-quasiregular.  Such a map is called uniformly quasiregular.
The bound in this theorem is sharp since the torus admits uniformly quasiregular self-mappings.

There are also related results when the manifold $M$ is open.  In dimension $2$, one can use the same arguments as in the compact case to deduce that $M$ is homeomorphic to $\bb R^2$ or $S^1\times \bb R$.
This result implies Picard's theorem as a corollary.  In higher dimensions, Rickman \cite{rickman1980} proved what is now known as the Rickman-Picard theorem, showing that a $K$-quasiregular map from $\bb R^d$ to the $d$-dimensional sphere $S^d$ can omit at most $C(d,K)$ points.  The fact that the constant depends on $K$ is unavoidable as seen in the constructions by Rickman \cite{rickman1985} and Drasin and Pankka\cite{drasinpankka}.

We next give an outline the proof for Theorem \ref{mainthm}.  We argue by contradiction.  Let $k > \binom{d}{l}$ and let $\alpha_1, \dots, \alpha_k$ be representatives of cohomology classes that form a basis in $H^l(M)$.  
Using Poincar\'e duality we can choose closed differential forms $\beta_1,\dots,\beta_k$ such that
\begin{align*}
    \int_M \alpha_i\wedge\beta_j =\delta_{ij},
\end{align*}
for $1 \le i,j\le k$ and where $\delta_{ij}$ is the Kronecker delta.
In previous papers on quasiregular ellipticity, $p$-harmonic forms were used instead of smooth forms arising from Poincar\'e duality.  Our approach allows us to avoid the use of this machinery.

Since we argue by contradiction, there exists a quasiregular mapping $f \colon \bb R^d \to M$.  The pullbacks, $\eta_i = f^*\alpha_i$ and $\theta_i = f^*(*\alpha_i)$ will be closed forms on $\bb R^d$.  They also satisfy local $L^p$-bounds depending on the Jacobian of $f$.  This allows us to use a rescaling procedure to obtain forms on the unit ball in $\bb R^d$ such that the limits wedge pointwise to $0$.

In the papers by Eremenko and Lewis, \cite{eremenkolewis} and \cite{lewis}, the authors applied a similar rescaling to $\cc A$-harmonic functions in order to prove the Rickman-Picard theorem for quasiregular mappings.  Instead of rescaling functions, we consider pullbacks of differential forms.  We also note that Kangasniemi \cite{kangasniemi} rescaled differential forms in the uniformly quasiregular case.  The main connection between the techniques used in this paper and the above two results is that in the limit the rescaled objects obey pointwise results.
This is the crucial ingredient of the proof.  
The rescaling captures how the map $f \colon \bb R^d \to M$ behaves on average.  Since quasiregular maps have equidistribution properties similar to holomorphic mappings, $f$ will map a large set evenly over $M$.  So the pullbacks of the differential forms, rescaled on a sequence of large balls, will converge to averages of themselves on $M$ .
The limits in this rescaling will be both nonzero and pair to $0$ pointwise; on the manifold the wedge product only integrates to $0$.

Once the differential forms on the unit ball are constructed and we know that they pair pointwise to $0$,  we see that at most $\binom{d}{l} =\dim (\bigwedge^l \bb R^d)$ of the forms can be nonzero.  This will imply that the sets where at least one of the forms is $0$ covers the entire ball, apart from a set of measure $0$.  However, the size of the rescaled forms is governed by the size of the Jacobian of $f$.  
In order to prove this we need to first show that the Jacobian of $f$ satisfies a reverse H\"older inequality.  In general, the Jacobian of a quasiregular mapping is in $L^1_{\text{loc}}(\bb R^d)$.  Bojarski and Iwaniec \cite{bojarskiiwaniec}, using a method similar to Gehring's lemma \cite{gehring}, showed that if $f\colon \bb R^d \to \bb R^d$, then the Jacobian of $f$ is in $L^{1+\epsilon}_{\text{loc}}(\bb R^d)$ for a sufficiently small $\epsilon$.  In addition, they show that $f$ satisfies a reverse H\"older inequality.  If $f\colon \bb R^d \to M$, then  the Jacobian of $f$ will be in $L^{1+\epsilon}_{\text{loc}}(\bb R^d)$, but it will not necessarily satisfy a reverse H\"older inequality.
The reverse H\"older inequality only holds when $H^l(M) \ne 0$ for some $l$, where $1\le l \le d-1$.

Once we know that the Jacobian of $f$ satisfies a reverse H\"older inequality, we prove that the size of the Jacobian governs the size of the rescaled forms.
In turn, this shows that the integral of the Jacobian of $f$ on a sequence of large balls will be arbitrarily small.
At this point we arrive at a contradiction since the balls were exactly chosen so that the integral of the Jacobian of $f$ is bounded away from $0$.  
Hence the number of forms is bounded by $\binom{d}{l}$.  These forms correspond to the dimension of the $l$-de Rham cohomology on $M$, proving Theorem \ref{mainthm}.

The structure of the paper is as follows.
Section \ref{forms} gives a brief introduction to differential forms on manifolds and pullbacks of differential forms by quasiregular mappings.  We also show the reverse H\"older inequality for the Jacobian of $f$.  For the relationship between quasiregular mappings and differential forms see \cite[Section 3]{bonkheinonen} and \cite{iwaniecmartin}.  The use of differential forms in this setting is inspired by the work of Bonk and Heinonen \cite{bonkheinonen}, Donaldson and Sullivan \cite{donaldsonsullivan} and Iwaniec and Martin \cite{iwaniecmartin}.

In Section \ref{rescale} we  discuss the rescaling argument and prove certain required convergence results.  Section \ref{mainproof} gives the proof of Theorem \ref{mainthm}.  Some of the methods in the proof are influenced by techniques developed by Pankka \cite{pankka}. For a reference on the facts used for quasiregular mappings see \cite{bonkheinonen}, \cite{donaldsonsullivan} and \cite{rickman1993}.

\subsection{Acknowledgments}
The author thanks Mario Bonk for both introducing him to the problem and the many discussions and comments on the paper.  The author would also like to thank Pekka Pankka for conversations in Helsinki on this topic.

\section{Exterior Algebra and Differential Forms}\label{forms}
This section gives an introduction to the tools needed to prove Theorem \ref{mainthm}. 

The space $M$ will always be a closed, connected and orientable Riemannian manifold of dimension $d$. 
Let $\bigwedge^l(\bb R^d)$ denote the space of degree $l$ exterior powers of the cotangent bundle of $\bb R^d$, for $1\le l \le d-1$.
By $\Omega^l(M)$, we mean the space of smooth differential forms on $M$ of degree $l$.  The de Rham cohomology of $M$ will be denoted by $H^l(M)$.  
Let $D\subset \bb R^d$, we say a differential form $\alpha$ is in $L^p(D)$, whenever the component functions of $\alpha$ are in the usual $L^p$-space.
Similarly, $\alpha$ is in the Sobolev space $W^{1,p}(D)$ whenever the component functions are in the standard Sobolev space, i.e.,  $\alpha_i \in L^p(D)$ and $\alpha_i$ has weak derivatives in $L^p(D)$.
On $\Omega^l(M)$, there exists an inner product induced by the Riemmanian metric on $M$.  For $\omega \in \Omega^l(M)$, we denote $\|\omega\|_\infty$ to be the $L^\infty$-norm given by this inner product.
Let $C_c^\infty(D)$ denote the space of smooth functions with compact support in $D$.
The exponents $p$ and $q$ will always denote $d/l$ and $d/(d-l)$ respectively. For $x \in \bb R^d$ and $r >0$, the set $B(x,r) \subset \bb R^d$ denotes the ball of radius $r$, centered at $x$.

In the following we can consider $l$ such that $1 \le l \le d-1$.  This is because $H^d(M)\cong H^0(M) \cong \bb R$ for the manifolds considered in Theorem \ref{mainthm}.

We will use Poincar\'e duality (see \cite[p. 44]{botttu}) to pick differential forms on $M$.
\begin{theorem}\label{poincare}
Let $k = \dim H^l(M)$, then there exists forms $\alpha_1,\dots,\alpha_k \in \Omega^l(M)$ and $\beta_1,\dots,\beta_k \in \Omega^{d-l}(M)$ such that $\{[\alpha_i]\}_{i=1}^k$ forms a basis for $H^l(M)$, $d\alpha_i = 0$, $d\beta_i = 0$ and
\begin{align}\label{orthogonal}
    \int_M \alpha_i\wedge\beta_j = \delta_{ij},
\end{align}
for $1 \le i,j\le k$.
\end{theorem}

We will often want to estimate integrals of certain differential forms. The following inequality will be useful later on.
If $\alpha \in \bigwedge^{l_1}(R^d)$ and $\beta \in \bigwedge^{l_2}(\bb R^d)$, then
\begin{align}\label{wedgeestimate}
    |\alpha \wedge \beta| \le C(d)|\alpha||\beta|,
\end{align}
where $C(d)$ only depends on the dimension.
To prove this note that the product $\alpha \wedge \beta$ is a bilinear operator on two finite dimensional vector spaces when $x$ is fixed.  Therefore it is bounded and we arrive at \eqref{wedgeestimate}.

A key tool we use is the pullback of a differential form by a quasiregular map.
If $f\colon\bb R^d \to M$ is quasiregular and $\omega \in \Omega^l(M)$, then
\begin{align}\label{stillclosed}
    d(f^*\omega) = f^*(d\omega)= 0.
\end{align}
We have that $f^*\omega \in L^{p}_{\text{loc}}(\bb R^d)$.  As a result, $d(f^*\omega)$ must be interpreted in the weak sense.  For a thorough discussion of this, see \cite[Section 2]{donaldsonsullivan}.

The next proposition gives a pointwise bound for these pullbacks.
\begin{prop}\label{pullbackestimate}
If $f \colon \bb R^d \to M$ is quasiregular and $\omega \in \Omega^l(M)$, then, for almost every $x \in \bb R^d$,
\begin{align*}
    |f^*\omega(x)| \le C(d)\|\omega\|_{\infty}\|Df(x)\|^l,
\end{align*}
where $\|Df\|$ is the operator norm for $Df$ and $C(d) > 0$ is a constant that depends only on $d$.
\end{prop}
\begin{proof}
The inequality we are trying to prove is a pointwise estimate.  So without loss of generality we may assume that $\omega \in \Omega^l(B(0,1))$.  For almost every $x \in \bb R^d$,
\begin{align*}
    f^*\omega(x) = \sum_{I} (\omega_I\circ f(x)) df^I(x)
\end{align*}
where $I = \{i_1,\dots,i_l\}$ is a multi-index of length $l$.  That is,
\begin{align*}
    df^I = df_{i_1}\wedge\cdots\wedge df_{i_l},
\end{align*}
where $f_i$ is $i$-th component function of $f$ and we sum  over all multi-indices, $1 \le i_1 < \cdots < i_l \le d$.
By Hadamard's inequality,
\begin{align*}
    |df_{i_1}\wedge\cdots\wedge df_{i_l}| &\le  |df_{i_1}|\cdots| df_{i_l}| \\
    &\le \|Df\|^l.
\end{align*}
Thus,
\[
    |f^*\omega(x)| \le C(d)\|\omega\|_{\infty}\|Df(x)\|^l.\qedhere
\]

\end{proof}

Bojarski and Iwaniec \cite{bojarskiiwaniec} showed that a quasiregular map $f \colon \bb R^d \to \bb R^d$ has a Jacobian that satifies a reverse H\"older inequality.  If $F,\Omega \subset \bb R^d$ are sets such that $F$ is compact, $\Omega$ is open and $F\subset \Omega$, then
\begin{align}
    \biggl (\int_{F}J_f^b\biggr )^{1/b} \le C(d,b,K)\frac{1}{\operatorname{dist}(F,\partial \Omega)^{d/a}} \int_{\Omega} J_f
\end{align}
where $\frac{1}{a} + \frac{1}{b} = 1$.  Crucially, $C(d,b,K)$ is independent of $f,F$ and $\Omega$.  They prove this by showing a weaker reverse H\"older inequality, where the exponents are $1$ and $1/2$.  They then use Gehring's lemma to upgrade to the above inequality.
We would like to have such a statement for $f \colon \bb R^d \to M$.  If $H^l(M) = 0$ for $1 \le l\le d-1$, then the Jacobian of $f$ does not necessarily satisfy a reverse H\"older inequality.  In our case there exists an $l$ such that $H^l(M) \ne 0$.
\begin{prop}\label{reverseholder}
Let $M$ be a closed Riemannian manifold and let  $f \colon \bb R^d \to M$ be $K$-quasiregular. If there exists an integer $l$ with $1\le l \le d-1$ such that $H^l(M) \ne 0$, then the Jacobian of $f$ satisfies the weak reverse H\"older inequality,
\begin{align*}
    \frac{1}{|\frac{1}{2}B|}\int_{\frac{1}{2}B} J_f \le  C(d,M,K) \biggl ( \frac{1}{|B|} \int_B J_f^{d/(d+1)} \biggr)^{(d+1)/d},
\end{align*}
where $B \subset \bb R^d$ is an arbitrary ball.
\end{prop}
\begin{proof}
Since $H^l(M) \ne 0$ there exists a Poincar\'e pair, $\alpha$ and $\beta$, given in Theorem \ref{poincare} with
\begin{align*}
    \int_M\alpha\wedge\beta = 1.
\end{align*}
This implies that there exists a point $a\in M$ so that for every chart $U$ around $a$,
\begin{align*}
    \alpha\wedge\beta|_x = g(x)dx^1\wedge\cdots\wedge dx^d,
\end{align*}
where $g(x) > 0$ and $x \in U$.
Let $x \in M$, by the Isotopy lemma \cite[p. 142]{guilleminpollack} there exists an orientation preserving diffeomorphism $\Phi_{x} \colon M \to M$ such that $\Phi_{x}(a) = x$.
Let $U$ be an open neighborhood around $a$ such that $\alpha\wedge\beta$ is positive in the sense above.  Then $( \Phi_x(U))_{x\in M}$ is an open cover of $M$ and there exists a finite subcover, $U_1,\dots,U_m$.
Since $\Phi_{x}$ is orientation preserving, $\Phi_x^*(\alpha\wedge \beta)$ is positive on $U_x$.  Let $\Phi_\nu$ be the diffeomorphism corresponding to $U_\nu$ and let $\{\lambda_\nu\}$ be a partition of unity subordinate to $\{U_\nu\}$.  Define
\begin{align*}
    \omega := \sum_{\nu = 1}^m \lambda_\nu \Phi_\nu^*(\alpha \wedge \beta).
\end{align*}
From this definition we get that for each chart on M,
\begin{align*}
    \omega|_x = h(x)dx^1\wedge \dots \wedge dx^d,
\end{align*}
where
\begin{align*}
    h(x) = \sum_{\nu =1}^m \lambda_\nu(x) g(\Phi_\nu(x)) J_{\Phi_\nu}(x).
\end{align*}
The diffeomorphism $\Phi_\nu$ is orientation preserving, so $J_{\Phi_\nu}(x) > 0$.  The functions $\lambda_\nu$ are always positive and only nonzero on $U_\nu$.  On the set $U_\nu$, $g(\Phi_\nu(x))$ is also positive.  So $h(x) > 0$.

The $d$-form $\omega$ is nonzero and so must be comparable to the volume form on $M$.  That is,
\begin{align*}
    \omega = cV,
\end{align*}
where $c \colon M \to \bb (0,\infty)$ is a positive, smooth function on $M$ and $V$ is the volume form on $M$.

With this preliminary representation of $V$ we can now proceed in showing the proposition.
Let $\psi \in C_c^\infty(\bb R^d)$ be a bump function that is $1$ on $\frac{1}{2}B$ and $0$ outside $B$.
The pullback of $V$ by $f$ is the Jacobian of $f$.  So
\begin{align*}
    \int_{\frac{1}{2}B} J_f & \le \int_B \psi J_f\\
    & = \int_{B} \psi f^*V \\
    & = \sum_{\nu = 1}^m\int_B\psi \frac{1}{c \circ f} (\lambda_\nu \circ f) f^*(\alpha_\nu \wedge \beta_\nu),
\end{align*}
where $\alpha_\nu = \Phi_\nu^*\alpha$ and $\beta_\nu = \Phi_\nu^*\beta$.
Since $m$ depends only on $M$ it suffices to bound a single term in the sum.  We also know that $1/c$ and $\lambda_\nu$ are positive and bounded above by constants depending only on $M$.  So it suffices to consider the integral,
\begin{align*}
    \int_B \psi f^*\alpha_\nu \wedge f^*\beta_\nu.
\end{align*}
On $M$, $\alpha_\nu$ is closed.  By \eqref{stillclosed}, $f^*\alpha = du$ on $B$.  Integration by parts gives that
\begin{align*}
     \biggl| \int_B \psi f^*\alpha_\nu \wedge f^*\beta_\nu \biggr| = \biggl |\int_B d\psi \wedge u\wedge f^*\beta_\nu \biggr|.
\end{align*}
By \eqref{wedgeestimate}, H\"older's inequality and because $|d\psi| \le \frac{1}{r}$, where $r$ is the radius of $B$,
\begin{align*}
    \biggl |\int_B d\psi \wedge u\wedge f^*\beta_\nu \biggr| \le \frac{C(d)}{r}\|u\|_{d^2/(l(d+1)-d)}\|f^*\beta\|_{d^2/((d+1)(d-l))}.
\end{align*}
Note that these exponents add up correctly in this inequality because $1\le l \le d-1$.  We can choose $u$ so that $u$ satisfies a Poincar\'e-Sobolev inequality.  For a precise formulation of this, see \cite[Corollary 4.2]{iwanieclutoborski}.  Since $du = f^*\alpha_\nu$,
\begin{align*}
    \frac{C(d)}{r}\|u\|_{d^2/(l(d+1)-d)}\|f^*\beta\|_{d^2/((d+1)(d-l))} \le \frac{C(d)}{r}\|f^*\alpha_\nu\|_{d^2/(l(d+1))}\|f^*\beta\|_{d^2/((d+1)(d-l))}.
\end{align*}
Again, we remark that the Poincar\'e-Sobolev inequality is only valid here because $1\le l \le d-1$.
The forms $\alpha_\nu$ and $\beta_\nu$ are smooth on $M$ and therefore are bounded independently of $f$.  So by \eqref{wedgeestimate},
\begin{align*}
    \frac{C(d)}{r}\|f^*\alpha_\nu\|_{d^2/(l(d+1))}\|f^*\beta\|_{d^2/((d+1)(d-l))} &\le \frac{C(d,M,K)}{r}\|J_f\|_{d/(d+1)}^{l/d}\|J_f\|_{d/(d+1)}^{(d-l)/d} \\
    & = \frac{C(d,M,K)}{r}\biggl (\int_B J_f^{d/(d+1)}\biggr )^{(d+1)/d}.
\end{align*}
We sum over $\nu$ and take averages to arrive at the proposition.
\end{proof}
Now that we have shown Proposition \ref{reverseholder}, \cite[Theorem 4.2]{bojarskiiwaniec} implies the following statement:
\begin{prop}\label{improvedholder}
Let $B \subset \bb R^d$ be a ball.  There exists $b > 1$ such that
\begin{align*}
\biggl (\frac{1}{|\frac{1}{2}B|}\int_{\frac{1}{2}B} J_f^b\biggr)^{1/b} \le C(d,M,K,b) \frac{1}{|B|}\int_{B}J_f.
\end{align*}
\end{prop}

\section{Rescaling Principle}\label{rescale}
In this section we construct rescaled forms on $B(0,1)$.
By Theorem \ref{poincare}, there exist closed differential forms $\alpha_1,\dots,\alpha_k \in \Omega^{l}(M)$ and $\beta_1,\dots,\beta_k \in \Omega^{d-l}(M)$ such that the cohomology classes $[\alpha_1],\dots,[\alpha_k]$ form a basis for $H^l(M)$.
In addition, they satisfy the orthogonality relation
\begin{align*}
    \int_M \alpha_i\wedge\beta_j = \delta_{ij},
\end{align*}
for $1 \le i,j \le k$.
We will rescale the pullbacks, $\eta_i = f^*\alpha_i$ and $\theta_i = f^*\beta_i$.
By \eqref{stillclosed}, $\eta_i$ and $\theta_i$ are closed.  By the quasiregularity of $f$, we have that $f \in W_{\text{loc}}^{1,d}(\bb R^d,M)$.  By Proposition \ref{pullbackestimate}, $\eta_i \in L^p_{\text{loc}}(\bb R^d)$ and $\theta_i \in L^q_{\text{loc}}(\bb R^d)$, where $p= d/l$ and $q = d/(d-l)$.
For $n \in \bb N$, let $\{B_n\}$ be a collection of balls in $\bb R^d$ that will be chosen below.  Define $T_n \colon B(0,1) \to B_n := B(a_n,r_n)$ as $T_n (x) : = a_n + r_nx$.  Next, we construct our rescaled forms as
\begin{align}\label{etan}
    \eta_i^n := \frac{1}{A(B_n)^{1/p}} T_n^*\eta_i
\end{align}
and
\begin{align}\label{thetan}
    \theta_i^n := \frac{1}{A(B_n)^{1/q}} T_n^*\theta_i,
\end{align}
where
\begin{align*}
    A(B) := \int_B J_f,
\end{align*}
for a Borel set $B \subset \bb R^d$.  Explicitly, if
\begin{align*}
    \eta_i = \sum_{I} h_I(x) dx^I,
\end{align*}
where the summation is over all $I = \{i_1,\dots,i_l\}$, multi-indices of length $l$ and where $dx^I = dx^{i_1}\wedge \cdots\wedge dx^{i_l}$.  We have that
\begin{align*}
    \eta_i^n = \frac{r_n^{d/p}}{A(B_n)^{1/p}} \sum_{I} h_I(a_n + r_nx) dx^I.
\end{align*}
Similarly,
\begin{align*}
    \theta_i^n = \frac{r_n^{d/q}}{A(B_n)^{1/q}} \sum_{J} g_J(a_n + r_nx) dx^J,
\end{align*}
where $J$ is a multi-index of length $(d-l)$.

The following theorem \cite[Theorem 1.11]{bonkheinonen} shows that $A(B(0,r))$ is unbounded.
\begin{theorem}\label{aunbounded}
Let $f \colon \bb R^d \to M$ be a quasiregular mapping.  If $H^l(M) \ne \{0\}$ for $1\le l < d$, then there exists a constant $\alpha >0$ such that
\begin{align*}
    \liminf_{r\to \infty} \frac{A(B(0,r))}{r^\alpha} > 0.
\end{align*}
In particular, $A(\bb R^d) = \infty$.
\end{theorem}
We also record a lemma due to Rickman (for the proof see \cite[Lemma 5.1]{rickman1980}),
\begin{lemma}[Rickman's Hunting Lemma]\label{hunting}
Let $\mu$ be a Borel measure on $\bb R^d$ that is absolutely continuous with respect to Lebesgue measure. If $\mu(\bb R^d) = \infty$, then, for all $M > 0$, there exists a point $a \in \bb R^d$ and a radius $r > 0$ such that
\begin{align*}
    \mu(B(a,r)) \ge M \quad \text{and} \quad \mu(B(a,r)) \le D(d) \mu(B(a,r/2)),
\end{align*}
where $D(d)$ is a constant that depends only on the dimension.
\end{lemma}
So by Theorem \ref{aunbounded} and Lemma \ref{hunting}, there exist balls $B_n \subset \bb R^d$ such that $\displaystyle\lim_{n\to \infty} A(B_n) = \infty$ and
\begin{align}\label{doubling}
    A(B_n) \le D(d)A (\tfrac{1}{2}B_n).
\end{align}
We use these balls in our definition of $\eta_i^n$ and $\theta_i^n$.

\begin{lemma}\label{conv1}
For $n\in \bb N$, there exists a $(d-l-1)$-form $u_i^n \in W^{1,q}(B(0,1))$, where $q= d/(d-l)$, such that
\begin{align*}
    du_i^n = \theta_i^n.
\end{align*}
Furthermore, we can pass to a subsequence so that the following convergence results hold.
\begin{enumerate}
    \item[\textnormal{(i)}] There exists an $l$-form $\tilde \eta_i \in  L^p(B(0,1))$ and a $(d-l)$-form $\tilde \theta_i \in L^q(B(0,1))$ such that
    \begin{align*}
        \lim_{n\to\infty} \eta_i^n = \tilde \eta_i \quad \text{and} \quad \lim_{n\to \infty} \theta_i^n = \tilde \theta_i
    \end{align*}
    where the convergence of $\eta_i^n$ is in the weak topology on $L^p(B(0,1))$ and the convergence of $\theta_i^n$ is in the weak topology on $L^q(B(0,1))$.
    \smallskip
    \item [\textnormal{(ii)}] There exists a $(d-l-1)$-form, $\tilde u_i \in W^{1,q}(B(0,1))$ such that
    \begin{align*}
        \lim_{n\to \infty} u_i^n = \tilde u_i
    \end{align*}
    in $L^q(B(0,1))$.
    \smallskip
    \item [\textnormal{(iii)}] On $B(0,1)$
    \begin{align*}
        d\tilde u_i = \tilde \theta_i
    \end{align*}
    in the weak sense.
\end{enumerate}
\end{lemma}
\begin{proof}
In the following proof we will often pass to subsequences.  It is understood that the subsequences should be taken simultaneously for all the forms mentioned in the lemma.

For the proof of (i), we compute the $L^p$-norm of $\eta_i^n$.  Indeed, by Equation \eqref{etan},
\begin{align*}
    \int_{B(0,1)}|\eta_i^n|^p &= \frac{r_n^d}{A(B_n)}\int_{B(0,1)} |\eta_i(a_n+r_nx)|^p \\
    & = \frac{1}{A(B_n)} \int_{B_n} |\eta_i|^p.
\end{align*}
By the quasiregularity of $f$ and Proposition \ref{pullbackestimate},
\begin{align*}
    \frac{1}{A(B_n)} \int_{B_n} |\eta_i|^p \le KC(d)\frac{\|\alpha_i\|^p_{\infty}}{A(B_n)}\int_{B_n}J_f \\
    &\le KC(d)\|\alpha_i\|^p_{\infty}.
\end{align*}
Hence, the $L^p$-norm of the $\eta_i^n$ is uniformly bounded.  By the Banach-Alaoglu theorem, we can pass to a subsequence so that 
\begin{align*}
    \lim_{n \to \infty} \eta_i^n = \tilde \eta_i,
\end{align*}
weakly in $L^p(B(0,1))$.

The proof for $\theta_i^n$ is very similar.  By $\eqref{thetan}$,
\begin{align*}
    \int_{B(0,1)}|\theta_i^n|^q &= \frac{r_n^d}{A(B_n)}\int_{B(0,1)} |\theta_i(a_n+r_nx)|^q \\
    & = \frac{1}{A(B_n)} \int_{B_n} |\theta_i|^q \\
    & \le KC(d)\frac{\|\beta_i\|^q_{\infty}}{A(B_n)}\int_{B_n}J_f\\
    &\le KC(d)\|\beta_i\|^q_{\infty}.
\end{align*}
Again, by the Banach-Alaoglu theorem, we can pass to a subsequence so that
\begin{align*}
    \lim_{n \to \infty} \theta_i^n = \tilde \theta_i
\end{align*}
weakly in $L^q(B(0,1))$.

We next prove (ii). By part (i), the $L^q$-norm of $\theta_i^n$ is uniformly bounded.
The forms $\theta_i^n$ are closed by \eqref{stillclosed}.
By the Sobolev embedding theorem, there exists $(d-l-1)$-forms, $u_i^n \in W^{1,q}(B(0,1))$ such that $du_i^n = \theta_i^n$ and $\|u_i^n\|_{d/(d-l-1)} \le C \|\theta_i^n\|_q$, where $C$ does not depend on $n,u_i^n$ or $\theta_i^n$ (see \cite[Corollary 4.2]{iwanieclutoborski}, for the formulation of the Sobolev embedding theorem and the Sobolev-Poincar\'e inequaliy for differential forms).
Furthermore, there exists a subsequence of $u_i^n$ that converges to $\tilde u_i$ strongly in $L^q(B(0,1))$.  We will also denote this subsequence as $u_i^n$.

Finally, we show (iii).
We demonstrate that $d\tilde u_i = \tilde\theta_i$ in the weak sense.
By duality, we can consider test forms $\phi \in \Omega^{l+1}(B(0,1))$ with compact support.  We pair $\tilde u_i$ with $d\phi$,
\begin{align*}
    \int_{\bb R^d} \tilde u_i \wedge d\phi &= \lim_{n\to \infty}\int_{\bb R^d} u_i^n\wedge d\phi \\
    & = \lim_{n\to \infty}(-1)^{d-l}\int_{\bb R^d} \theta_i^n\wedge \phi\\
    & = (-1)^{d-l}\int_{\bb R^d} \tilde \theta_i \wedge \phi.
\end{align*}
This proves the claims in the lemma.
\end{proof}

We need one more convergence result.
\begin{lemma}\label{conv2}
Let $\psi \in C_c^\infty(B(0,1))$.  Then
\begin{align*}
    \lim_{n\to \infty} \int_{B(0,1)} \psi \eta_i^n\wedge\theta_j^n  = \int_{B(0,1)} \psi \tilde\eta_i\wedge \tilde \theta_j,
\end{align*}
for $1\le i,j \le k$.
\end{lemma}
\begin{proof}
Consider the difference,
\begin{align*}
    \biggl |\int_{B(0,1)}\psi \eta_i^n\wedge\theta_j^n  - \int_{B(0,1)} \psi \tilde\eta_i\wedge \tilde \theta_j  \biggr | &\le \biggl |\int_{B(0,1)}\psi \eta_i^n\wedge (\theta_j^n-\tilde \theta_j) \biggr |\\
    &+ \biggl |\int_{B(0,1)}\psi (\eta_i^n - \tilde \eta_i)\wedge\tilde \theta_j \biggr | \\
    & = I + II.
\end{align*}
Lemma \ref{conv1} gives that
\begin{align*}
     I &=  \biggl |\int_{B(0,1)} \psi \eta_i^n\wedge (du_j^n- d\tilde u_j) \biggr|.
\end{align*}
By integration by parts and the compact support of $\psi$,
\begin{align*}
     \int_{B(0,1)} \psi \eta_i^n\wedge d(u_j^n -\tilde u_j)&=(-1)^{l+1}\int_{B(0,1)}  d(\psi\eta_i^n)\wedge (u_j^n-\tilde u_j)\\
     & = (-1)^{l+1}\int_{B(0,1)}  d\psi \wedge \eta_i^n\wedge (u_j^n  -\tilde u_j)
\end{align*}
because $\eta_i^n$ is weakly closed and $\psi(u_j^n-\tilde u_j) \in W^{1,q}(\bb R^d)$. By \eqref{wedgeestimate},
\begin{align*}
     |d\psi \wedge \eta_i^n\wedge (u_j^n  -\tilde u_j)| \le C(d) |d\psi\wedge\eta_i^n||u_j^n  -\tilde u_j|,
\end{align*}
where $C(d)$ only depends on $d$.
By H\"older's inequality,
\begin{align*}
    I \le  C(d)\|d\psi\wedge\eta_i^n\|_p\|u_j^n - \tilde u_j\|_q.
\end{align*}
By Lemma \ref{conv1}, the term $\|d\psi\wedge\eta_i^n\|_p$ is bounded independently of $n$.  Lemma \ref{conv1} also gives that $u_i^n \to \tilde u_i$ in $L^q(B(0,1))$.  So  $\lim_{n\to \infty} |I| = 0$.  For the term $II$, by Lemma \ref{conv1}, $\eta_i^n \to \tilde \eta_i$ in $L^p(B(0,1))$ in the weak sense.  In addition, $\psi \tilde \theta_j \in L^q(B(0,1))$.  It follows that 
\begin{align*}
    \lim_{n\to \infty}II = \lim_{n\to \infty} \biggl |\int_{B(0,1)} (\eta_i^n - \tilde \eta_i)\wedge ( \psi \tilde \theta_j) \biggr |=0.
\end{align*}

\end{proof}

\section{Proof of Theorem \ref{mainthm}}\label{mainproof}
In this section we complete the proof of the main result.
\begin{lemma}
Let $\tilde \eta_i$ and $\tilde \theta_i$ be the forms constructed in Section \ref{rescale}.
For almost every $x\in B(0,1)$,
\begin{align}\label{zeropair2} 
    \tilde \eta_i \wedge \tilde \theta_j(x) = 0
\end{align}
when $i \ne j$.
\end{lemma}
\begin{proof}
When $i\ne j$,
\begin{align*}
    \int_M \alpha_i\wedge \beta_j = 0,
\end{align*}
by \eqref{orthogonal}.  
By de Rham's theorem \cite[Corollary 5.8]{botttu}, there exists $\tau \in \Omega^{d-1}(M)$ such that $d\tau = \alpha_i\wedge\beta_j$.
Let $\psi \in C_c^\infty(B(0,1))$, using integration by parts and the compact support of $\psi$,
\begin{align*}
    \int_{B(0,1)}\psi \eta_i^n\wedge \theta_j^n &= \frac{1}{A(B_n)}\int_{B_n} \psi\biggl (\frac{x-a_n}{r_n}\biggr ) d(f^*\tau)(x) \\
    & = \frac{-1}{A(B_n)}\int_{B_n} d\biggl (\psi\biggl (\frac{x-a_n}{r_n}\biggr )\biggr)\wedge f^*\tau(x).
\end{align*}
By \eqref{wedgeestimate} and H\"older's inequality,
\begin{align*}
    \biggl |\int_{B(0,1)}\psi \eta_i^n\wedge \theta_j^n \biggr | \le \frac{C(d)}{A(B_n)}\|d\psi\|_{d,B(0,1)}\biggl( \int_{B_n} |f^*\tau|^{d/(d-1)} \biggr )^{(d-1)/d}.
\end{align*}
By Proposition \ref{pullbackestimate} and the quasiregularity of $f$,
\begin{align*}
    \biggl |\int_{B(0,1)}\psi \eta_i^n\wedge \theta_j^n\biggr | \le C(d)K^{(d-1)/d}\frac{\|d\psi\|_{d,B(0,1)}\|\tau\|_{\infty} }{A(B_n)} \biggl(\int_{B_n} J_f\biggr)^{(d-1)/d}.
\end{align*}
So
\begin{align*}
    \biggl |\int_{B(0,1)}\psi \langle \eta_i^n,\theta_j^n\rangle \biggr | \le C(K,M,d)\frac{\|d\psi\|_{d,B(0,1)}}{A(B_n)^{1/d}} \to 0
\end{align*}
as $n\to \infty$.
By Lemma \ref{conv2}, 
\begin{align*}
    \int_{B(0,1)}\psi\tilde \eta_i\wedge \tilde \theta_j =0.
\end{align*}
Since $\psi$ was an arbitrary test function,
$\tilde \eta_i\wedge \tilde \theta_j(x) = 0$ for almost every $x\in B(0,1)$.
\end{proof}

We have assumed that $k > \binom{d}{l}$.
This implies that for almost every $x \in B(0,1)$ there exists an $i \in \{1,\dots,k\}$ such that 
\begin{align}\label{zeronorm}
    \tilde \eta_i\wedge \tilde \theta_i(x)  = 0.
\end{align}
To see this, fix $x \in B(0,1)$ such that \eqref{zeropair2} holds for all pairs.
Let $\{\tilde \eta_{i_1}(x),\dots, \tilde \eta_{i_m}(x)\}$ be a basis for $\operatorname{span}(\{\tilde \eta_i(x)\}_{i=1}^k)\subset\bigwedge^l \bb R^d$.  Since dimension of $\bigwedge^l\bb R^d$ is $\binom{n}{l}$, we have that $m \le \binom{n}{l}$.  By our assumption $k > \binom{d}{l}$, so there exists a form $\tilde \eta_j \notin \{\tilde \eta_{i_1}(x),\dots, \tilde \eta_{i_m}\}$.
It follows that
\begin{align*}
    \tilde \eta_j\wedge \tilde \theta_j(x) &=  \sum_{a=1}^m\lambda_{i_a}  \tilde \eta_{i_a}\wedge \tilde \theta_j(x) \\
    & = 0
\end{align*}
by \eqref{zeropair2}. 

Therefore, for almost every $x\in B(0,1)$, one of the pairings $\tilde \eta_i\wedge \tilde \theta_i(x) $ must be $0$.
Let $D_i = \{x \in B(0,1) : \tilde \eta_i\wedge \tilde \theta_i(x) = 0\}$ and define $D_i^n = a_n + r_n D_i$.  Then $|B_n| = |\bigcup D_i^n|$ and
\begin{align*}
    A(\tfrac{1}{2}B_n) = \sum_{i=1}^k \int_{D_i^n\cap \frac{1}{2}B_n} J_f.
\end{align*}
For each $n\in \bb N$ there exists an $i$ so that
\begin{align*}
    \int_{D_i^n\cap \frac{1}{2}B_n} J_f \ge \frac{1}{k} A(\tfrac{1}{2}B_n) \ge \frac{A(B_n)}{kD(d)}.
\end{align*}
by \eqref{doubling}.  Taking a subsequence of the $n$ we can ensure that the $i$ is always the same.
\begin{lemma}\label{smallintegral}
For all $\epsilon >0$, there exists a compact set $C_i \subset D_i\cap B(0,\frac{1}{2})$ and an open set $E_i$ containing $D_i \cap B(0,\frac{1}{2})$  such that
\begin{align}\label{bigzero}
    \int_{C_i^n}J_f \ge \frac{A(B_n)}{2kD(d)},
\end{align}
where $C_i^n = a_n + r_nC_i$, and
\begin{align}\label{etasmall}
   \int_{E_i}|\tilde \eta_i\wedge \tilde \theta_i|< \epsilon.
\end{align}
\end{lemma}
\begin{proof}
Fix $\epsilon > 0$.
By outer regularity, there exists a set $E_i$ containing  $D_i\cap B(0,\frac{1}{2})$ such that \eqref{etasmall} is satisfied.

To construct $C_i$, first note that for all $\delta >0$, there exists compact sets $C_i(\delta) \subset D_i\cap B(0,\frac{1}{2})$ such that
\begin{align*}
|(D_i\cap B(0,\tfrac{1}{2}))\setminus C_i(\delta)| <\delta.
\end{align*}
Let $C_i^n(\delta) = a_n + r_nC_i(\delta)$ and $D_i^n = a_n + r_nD_i$.  To simplify notation, denote $\frac{1}{2} D_i :=D_i \cap B(0,\frac{1}{2})$ and $\frac{1}{2} D_i^n :=a_n+r_n\frac{1}{2}D_i$.
By H\"older's inequality,
\begin{align*}
    \int_{\frac{1}{2} D_i^n\setminus C_i^n(\delta)} J_f \le |\tfrac{1}{2} D_i^n\setminus C_i^n(\delta)|^{1/a}\biggl (\int_{\frac{1}{2} D_i^n\setminus C_i^n(\delta)} J_f^{b}\biggr )^{1/b}
\end{align*}
where $\frac{1}{a}+ \frac{1}{b} = 1$ and $b>1$ can be chosen arbitrarily close to $1$.  Continuing the calculation, we get
\begin{align*}
     \int_{\frac{1}{2} D_i^n\setminus C_i^n(\delta)} J_f &= r_n^{d/a}|\tfrac{1}{2} D_i\setminus C_i(\delta)|^{1/a}\biggl (\int_{\frac{1}{2} D_i^n\setminus C_i^n(\delta)} J_f^{b}\biggr )^{1/b} \\
    & \le  r_n^{d/a}|\tfrac{1}{2} D_i\setminus C_i(\delta)|^{1/a} \biggl (\int_{\frac{1}{2}B_n} J_f^{b}\biggr )^{1/b}.
\end{align*}
We now use the higher integrability for Jacobians of quasiregular mappings given in Proposition \ref{improvedholder},
\begin{align*}
     r_n^{d/a}|\tfrac{1}{2} D_i\setminus C_i(\delta)|^{1/a} \biggl (\int_{\frac{1}{2}B_n} J_f^{b}\biggr )^{1/b} & \le C(K,M,d,b)|\tfrac{1}{2} D_i\setminus C_i(\delta)|^{1/a} r_n^{d/a} r_n^{-d/a}\int_{B_n} J_f \\
    & =  C(K,M,d,b) |\tfrac{1}{2} D_i\setminus C_i(\delta)|^{1/a} A(B_n).
\end{align*}
We can choose $\delta$ to be arbitrarily small so that $|\frac{1}{2} D_i\setminus C_i(\delta)|^{1/a} < \frac{1}{2C(K,M,d,b) k D(d)}$.  This proves the lemma.
\end{proof}

We now have all of the ingredients to finish the proof for Theorem \ref{mainthm}.  Let $C_i,E_i$ be the sets given in Lemma \ref{smallintegral}. 
Define $C_i^n, E_i^n$ similarly as above.  Let $\psi \in C_c^\infty(B_n)$ and consider the following difference,
\begin{align*}
    \biggl | \int_{B_n} \psi^d  \eta_i\wedge\theta_i - \int_{B_n} \psi^d J_f  \biggr | = \biggl | \int_{B_n} \psi^df^*(\alpha_i\wedge\beta_i - V)\biggr |,
\end{align*}
where $V$ is the volume form on $M$.
We assume that $\operatorname{vol}(M) = 1$, so the $d$-form $\alpha_i\wedge\beta_i - V$ integrates to $0$ on $M$.  By de Rham's theorem, it is exact and $\alpha_i\wedge\beta_i - V = d\tau$, where $\tau \in \Omega^{d-1}(M)$.  We apply integration by parts and H\"older's inequality,
\begin{align*}
    \biggl | \int_{B_n} \psi^df^*(\alpha_i\wedge\beta_i - V)\biggr | &= \biggl | \int_{B_n} \psi^d d(f^*\tau) \biggr | \\
    & = \biggl |d\int_{B_n} \psi^{d-1}d\psi\wedge f^*\tau \biggr |.
\end{align*}
By \eqref{wedgeestimate} and H\"older's inequality,
\begin{align*}
     \biggl | \int_{B_n} \psi^df^*(\alpha_i\wedge\beta_i - V)\biggr |\le C(d) \|d\psi\|_{d,B_n}\biggl (\int_{B_n}\psi^d|f^*\tau|^{d/(d-1)}\biggr )^{(d-1)/d}.
\end{align*}
By Proposition \ref{pullbackestimate} and the quasiregularity of $f$, 
\begin{align*}
     \biggl | \int_{B_n} \psi^df^*(\alpha_i\wedge\beta_i - V)\biggr |\le C(d)K^{(d-1)/d}\|\tau\|_{\infty}^{(d-1)/d}d\|d\psi\|_{d,B_n}\biggl (\int_{B_n}\psi^dJ_f\biggr )^{(d-1)/d}.
\end{align*}
Dividing by $\int_{B_n}\psi^d J_f$ yields,
\begin{align*}
    \biggl | \frac{1}{(\int_{B_n}\psi^dJ_f)}\int_{B_n} \psi^d  \eta_i\wedge \theta_i - 1  \biggr | \le C(K,d,M)\|d\psi\|_{d,B_n}\biggl (\int_{B_n}\psi^dJ_f\biggr )^{-1/d}
\end{align*}
Choose $\psi$ so that $\psi >0, \psi \equiv 1$ on $C_i^n$ and $\psi \equiv 0$ outside $E_i^n$.  If $\tilde \psi = \psi((x-a_n)/r_n)$, then $\tilde \psi$ is a bump function that is $1$ on $C_i$ and $0$ outside $E_i$.  And
\begin{align*}
    \|d\psi\|_{d,B_n} = \|d\tilde\psi\|_{d,B(0,1)},
\end{align*}
since $C_i^n,E_i^n$ are conformally equivalent to $C_i,E_i$ respectively. In other words, the term with $\psi$ is independent of $n$.  This gives that
\begin{align}\label{nearone}
    \biggl | \frac{1}{(\int_{B_n}\psi^dJ_f)}\int_{B_n} \psi^d \eta_i\wedge\theta_i - 1  \biggr | \le C(K,d,M)  \|d\tilde\psi\|_{d,B(0,1)} A(B_n)^{-1/d},
\end{align}
which goes to $0$ as $n \to \infty$.

By Lemma \ref{conv2},
\begin{align*}
    \lim_{n\to \infty} \biggl |\frac{1}{A(B_n)}\int_{B_n} \psi^d \eta_i^n\wedge \theta_i^n\biggr | &=\biggl| \int_{B(0,1)} \tilde{\psi}^d\tilde \eta_i\wedge\tilde \theta_i\biggr | \\
    & \le \int_{B(0,1)} \tilde{\psi}^d |\tilde \eta_i\wedge \tilde \theta_i|
\end{align*}
Since the support of $\tilde \psi$ is contained in $E_i$,
\begin{align*}
    \int_{B(0,1)} \tilde{\psi}^d|\tilde \eta_i\wedge\tilde \theta_i| & \le  \int_{E_i} |\tilde \eta_i\wedge \tilde \theta_i|\\
    & < \epsilon,
\end{align*}
by \eqref{etasmall}.
So, for $n$ sufficiently large, we have that
\begin{align}\label{e1}
    \biggl |\frac{1}{A(B_n)}\int_{B_n} \psi^d \eta_i\wedge\theta_i\biggr | \le 2\epsilon.
\end{align}
By \eqref{bigzero},
\begin{align}\label{e2}
    \biggl (\int_{B_n}\psi^dJ_f\biggr )^{-1/d} \le \biggl (\int_{C_i^n}J_f\biggr )^{-1/d} \le (2D(d)k)^{1/d} A(B_n)^{-1/d}.
\end{align}
Therefore, using \eqref{e1} and \eqref{e2},
\begin{align*}
    \frac{1}{(\int_{B_n}\psi^dJ_f)}\biggl |\int_{B_n} \psi^d  \eta_i\wedge\theta_i\biggr | &= \frac{A(B_n)}{(\int_{B_n} \psi^dJ_f)}\biggl |\frac{1}{A(B_n)}\int_{B_n} \psi^d  \eta_i\wedge\theta_i\biggr | \\
    & \le \frac{A(B_n)}{(\int_{B_n} \psi^dJ_f)}2\epsilon \\
    & \le \frac{A(B_n)}{(\int_{C_i^n}J_f)}2\epsilon \\
    & \le 4kD(d)\epsilon.
\end{align*}
This bound is independent of $n$ and  contradicts \eqref{nearone} for small $\epsilon$ and large $n$.  Therefore $|\bigcup D_i| \ne |B(0,1)|$ and $k \le \binom{d}{l}$.  This proves Theorem \ref{mainthm}.

\end{document}